\documentclass[a4paper,12pt,oneside]{article}
\usepackage{inputenc}
\usepackage[english]{babel}
\usepackage{amsthm}
\usepackage{amsmath}
\usepackage{amssymb}
\usepackage{hyperref}
\usepackage{geometry}
\usepackage{cite}
\usepackage{enumerate}

\theoremstyle{plain}
\newtheorem{mydef}{Definition}
\newtheorem{myprop}[mydef]{Proposition}
\theoremstyle{plain}
\newtheorem{mythm}[mydef]{Theorem}
\newtheorem{mycor}[mydef]{Corollary}
\newtheorem{mylemma}[mydef]{Lemma}
\theoremstyle{remark}
\newtheorem{myex}[mydef]{Example}
\newtheorem{myremark}[mydef]{Remark}

\numberwithin{equation}{section}

\newcommand{\mbf}[1]{\mathbf{#1}}

\newcommand{\norm}[1]{\left\lVert #1 \right\rVert}

\newcommand{\dist}[2]{d\left( #1,#2 \right)}

\newcommand{\cone}{\textup{cone}}

\newcommand{\fora}{\textup{for all }}
\newcommand{\for}{\textup{for }}

\newcommand{\relconst}{\tilde\varepsilon}
\DeclareMathOperator*{\id}{Id}
\DeclareMathOperator*{\Fix}{Fix}

\newcommand{\N}{\mathbb{N}}
\newcommand{\R}{\mathbb{R}}
\newcommand{\Ball}{\mathbb{B}}

\newcommand{\E}{\mbf{E}}

\newcommand{\point}{\overline x}
\newcommand{\set}{\Omega}
\renewcommand{\equiv}{:=}
\newcommand{\theauthor}{Robert Hesse }
\newcommand{\thedate}{\today }
\newcommand{\thetitle}{Projection Methods for nonconvex feasibility problems}

\date{\thedate}
\author{\theauthor}
\title{\thetitle}
\begin{document}

\title{\textrm{Nonconvex notions of regularity and convergence of fundamental 
algorithms for feasibility problems
}}

\author{
Robert Hesse\thanks{Institut f\"ur Numerische und Angewandte Mathematik\
Universit\"at G\"ottingen,\ Lotzestr.~16--18, 37083 G\"ottingen, Germany. E-mail: \texttt{hesse@math.uni-goettingen.de}.} ~and
D.\ Russell Luke\thanks{Institut f\"ur Numerische und Angewandte Mathematik\
Universit\"at G\"ottingen,\ Lotzestr.~16--18, 37083 G\"ottingen, Germany. E-mail: \texttt{r.luke@math.uni-goettingen.de}.} 
}
\date{\today}

\maketitle

\vskip 8mm

\abstract{
We consider projection algorithms for solving (nonconvex) feasibility problems in Euclidean spaces.
Of special interest are the Method of Alternating Projections (MAP) and the Douglas-Rachford algorithm (DR).
In the case of convex feasibility, firm nonexpansiveness of projection mappings is a global property that yields 
global convergence of MAP and for consistent problems DR.
A notion of local sub-firm nonexpansiveness with respect to the intersection is introduced for consistent feasibility problems.
This, together with a coercivity condition that relates to the regularity of the collection of sets at points in the intersection,  
yields local linear convergence of MAP for a wide class of nonconvex problems, 
and even local linear convergence of nonconvex instances of the Douglas-Rachford algorithm. }
\bigskip
 
\bigskip



\section{Introduction}
In the last decade there has been significant progress in the understanding of 
convergence of algorithms for solving generalized equations for nonmonotone 
operators, and in particular those 
arising from variational problems such as minimization or maximization of functions, feasibility, 
variational inequalities and minimax problems.  Early efforts focused on the proximal point 
algorithm \cite{Pen02, Ius03, CombettesPennanen04} and notions of {\em (co)hypomonotonicity} which is closely related to 
{\em prox-regularity} of functions \cite{PolRockThib00}.  Other works have focused on {\em metric
regularity} and its refinements \cite{Ara05, AragonGeoffroy07, KlatteKummer09}.  Proximal-type algorithms
have been studied for functions satisfying the Kurdyka-\L ojasiewicz Inequality in \cite{AttouchBolteRedontSoubeyran}. 
In a more limited context, the method of alternating projections (MAP) has been investigated in \cite{LLM, BLPW2} 
with the aim of formulating {\em dual} characterizations of regularity requirements for linear convergence via the 
{\em normal cone} and its variants.  

The framework we present here 
generalizes the tools for the analysis of fixed-point iterations of operators that violate 
the classical property of firm nonexpansiveness in some quantifiable fashion.  
As such, our approach is more closely related to the 
ideas of \cite{Pen02} and hypomonotonicity through the resolvent mapping, however 
our application to MAP bears a direct resemblance to the more classical {\em primal} characterizations of 
regularity described in \cite{BauBorSIREV} (in particular what they call  ``linear regularity'').  
There are also some parallels between what we call $(S,\varepsilon)$-firm nonexpansiveness 
and $\varepsilon$-Enlargements of maximal monotone operators \cite{Burachikbook}, though this 
is beyond the scope of this paper.  Our goal is to introduce the essential tools we make use of with 
a cursory treatment of the connections to other concepts in the literature, and to apply these tools
to the MAP and Douglas-Rachford algorithms, comparing our results to the best known results at this time.  

We review the basic definitions and classical definitions and results 
below. In section \ref{s:S-epsilon-firm} we introduce our relaxations of the notions of firm-nonexpansiveness and 
set regularity.  Our main abstract result 
concerning fixed-point iterations of mappings that violate the classical firm-nonexpansive assumptions 
is in section \ref{s:convergence}.   We specialize in subsequent subsections to MAP and the 
Douglas-Rachford algorithm. 
Our statement of linear convergence of MAP is as general as the results reported in \cite{BLPW2}, with 
more elementary proofs, although our estimates for the radius of convergence are more conservative.  
The results on local linear convergence of nonconvex instances of Douglas-Rachford are new and 
provide some evidence supporting the conjecture  
that, asymptotically, Douglas-Rachford converges more slowly (albeit still linearly) than simple alternating projections.  Our estimates
of the rate of convergence for both MAP and Douglas-Rachford are not optimal, but to our knowledge the most general to date.  

We also show that strong regularity 
conditions on the collection of {\em affine} sets are in fact {\em necessary} for linear convergence of iterates of the 
Douglas-Rachford algorithm to the intersection, 
in contrast to MAP where the same conditions are sufficient, but not necessary \cite{BLPW2}.  
This may seem somewhat spurious to experts since, as is well-known, the Douglas-Rachford iterates 
themselves are not of interest, but rather their shadows or projections 
onto one of the sets \cite{LionsMercier, BauComLuke}.  Indeed, in the convex setting where local is global, the shadows of the iterates of 
the Douglas-Rachford algorithm could converge even though the iterates themselves diverge.  This happens in particular 
when the sets do not intersect, but have instead {\em best approximation points} \cite[Theorem 3.13]{BauComLuke}.  
The nonconvex setting is much less forgiving, however.  Indeed, existence of local best approximation points does 
not guarantee convergence of the shadows to best approximation points in the nonconvex setting  \cite{Luke08}, and 
so convergence of the sequence itself is essential.  As nonconvex settings are our principal interest, we focus 
on convergence of the iterates of the Douglas-Rachford algorithm instead of the shadows, and in particular convergence of these 
iterates to the intersection of collections of sets.  We leave a fuller examination of the shadow sequences to future 
work.  

\subsection{Basics/Notation}
$\E$ is a Euclidean space. 
We denote the closed unit ball centered at the origin by $\Ball$ and the ball of radius $\delta$ centered at $\point\in\E$ by 
$\Ball_\delta(\point):=\left\{ x\in\E \middle | \norm{x-\point}\leq\delta\right\}$.  
When the $\delta$-ball is centered at the origin we write $\Ball_\delta$. 
The notation ``$\rightrightarrows$'' indicates that this mapping in general is \emph{multi-valued}.
The composition of two \emph{multi-valued} mappings $T_1,T_2$ is pointwise defined by $T_2\ T_1 x=\cup_{y\in T_1 x}T_2 y$.
A nonempty subset $K$ of $\E$  is a cone if, for all $\lambda>0$,   $\lambda K \equiv\{ \lambda k~|~k \in  K\} \subseteq K$. 
The smallest cone containing a set $\set\subset\E$  is denoted $\cone(S)$.

\begin{mydef}
 Let $\set\subset\E$ be nonempty, $x\in\E$. The \emph{distance} of $x$ to $\set$ is defined by
\begin{align}
\dist x \set:=\inf_{y\in\set}\norm{x-y}.\label{eq:distance}
\end{align}
\end{mydef}
\begin{mydef}[projectors/reflectors]
Let $\set\subset\E$ be nonempty and $x\in\E$. The (possibly empty) set of all best approximation points from $x$ to $\set$ denoted $P_\set(x)$ (or $P_\set x$), is given by
\begin{align}
P_\set (x):=&\left\{ y\in\set~\middle|~\norm{x-y}=\dist x \set\right\} \label{eq:projection}.
\end{align}
 The mapping $P_\set\rightrightarrows\set$ is called the \emph{metric projector}, or {\em projector}, onto $\set$.  We 
call an element of  $P_\set(x)$ a {\em projection}. 
The reflector $R_\set:\E\rightrightarrows\E$ to the set $\set$ is defined as
\begin{align}
R_\set x:=2P_\set x- x, \label{eq:reflector}
\end{align}
 for all $x\in\E$.
\end{mydef}
Since we are on a Euclidean space $\E$ convexity and closedness of a subset $C\subset\E$ is sufficient for the projector (respectively the reflector) to be single valued.
Closedness of a set $\set$ suffices for the set $\set$ being \emph{proximinal}, i.e. $P_C x\neq\emptyset$ for all $x\in\E$ 
(For a modern treatment see \cite[Corollary 3.13]{BauCom}.

\begin{mydef}[Method of Alternating Projections]\label{MAP}
For two sets $A,B\subset\E$ we call the mapping 
\begin{align}\label{eq:MAP}
T_{MAP}x=P_A P_B x
\end{align}
the Method of Alternating Projections operator.  
We call the MAP algorithm, or simply MAP, the corresponding Picard iteration,
\begin{align}\label{alg:MAP}
x_{n+1}\in T_{MAP}x_n, \quad n=0,1,2,\dots
\end{align}
for $x_0$ given. 
\end{mydef}

\begin{mydef}[Averaged Alternating Reflections/Douglas Rachford]\label{AAR}
For two sets $A,B\subset\E$ we call the mapping 
\begin{align}
T_{DR}x=\frac 12\left(R_A R_B x+x\right)\label{eq:AAR}
\end{align}
the Douglas-Rachford operator.
We call the {\em Douglas-Rachford algorithm}, or simply {\em Douglas-Rachford}, the corresponding Picard iteration,
\begin{align}\label{alg:AAR}
x_{n+1}\in T_{DR}x_n, \quad n=0,1,2,\dots
\end{align}
for $x_0$ given. 
\end{mydef}

\begin{myex}
\label{examples}
The following easy examples will appear throughout this work and serve to illustrate the regularity concepts
we introduce and the convergence behavior of 
the algorithms under consideration.
\begin{enumerate}[{(i)}]
 \item\label{ex:lines2}
 Two lines in $\R^2$:
\begin{align} 
  A=&\left\{ (x_1,x_2)\in\R^2\mid x_2=0\right\}\subset\R^2\\
  B=&\left\{ (x_1,x_2)\in\R^2\mid x_1=x_2\right\}\subset\R^2.
 \end{align} 
We will see that MAP and Douglas-Rachford converge with a linear rate to the intersection.
\item \label{ex:lines3}
Two lines in $\R^3$:
\begin{align} 
  A=&\left\{ (x_1,x_2,x_3)\in\R^3\mid x_2=0,x_3=0\right\}\subset\R^3\\
  B=&\left\{ (x_1,x_2,x_3)\in\R^3\mid x_1=x_2,x_3=0\right\}\subset\R^3.
 \end{align} 
After the first iteration step MAP shows exactly the same convergence behavior as in the first example. Douglas-Rachford does not converge to $\{0\}=A\cap B$. 
All iterates from starting points on the line $\{t(0,0,1)~|~ t\in \R \}$ are fixed points of the Douglas Rachford operator.  On the other hand, iterates from 
starting points in $A+B$ stay in $A+B$, and the case then reduces to Example (\ref{ex:lines2}).
 \item\label{ex:lineball}
 A line and a ball intersecting in one point:
 \begin{align} 
  A=&\left\{ (x_1,x_2)\in\R^2\mid x_2=0\right\}\subset\R^2\\
  B=&\left\{ (x_1,x_2)\in\R^2\mid x_1^2+(x_2-1)^2\leq1\right\}.
 \end{align} 
 MAP converges to the intersection, but not with a linear rate. Douglas-Rachford has fixed points that lie outside the intersection.
 \item\label{ex:cross}
A \emph{cross} and a subspace in $\R^2$:
  \begin{align} 
   A=&\R\times\{0\}\cup \{0\}\times\R\\
  B=&\left\{ (x_1,x_2)\in\R^2\mid x_1=x_2\right\}.
 \end{align} 
 This example relates to the problem of \emph{sparse-signal recovery}.
 Both MAP and Douglas-Rachford converge globally to the intersection $\{0\}=A\cap B$, though $A$ is nonconvex. 
The convergence of both methods is covered by the theory built up in this work.
 \item\label{ex:Borwein} A circle and a line:
 \begin{align} 
  A=&\left\{ (x_1,x_2)\in\R^2\mid x_2=\sqrt{2}/2\right\}\subset\R^2\\
  B=&\left\{ (x_1,x_2)\in\R^2\mid x_1^2+x_2^2=1\right\}.
 \end{align} 
 This example is of our particular interest, since it is a simple model case of the \emph{phase retrieval problem}.
 So far the only {\em direct} nonconvex convergence results for Douglas-Rachford are related to this model 
case, see \cite{BorweinAragon,BorweinSims}.
 Local convergence of MAP is covered by \cite{LLM,BLPW2} as well as by the results in this work.
\end{enumerate}
\qed
\end{myex}

\subsection{Classical results for (firmly) nonexpansive mappings}
To begin, we recall (firmly) nonexpansive mappings and 
their natural association with projectors and reflectors on convex sets. 
We later extend this notion to nonconvex settings where the algorithms involve
set-valued mappings.
\begin{mydef}
Let $\set\subset\E$ be nonempty.
$T:\set\to\E$ is called \emph{nonexpansive}, if
\begin{align}
\norm{Tx-Ty}\leq\norm{x-y}\label{eq:nonexpansive}
\end{align}
holds for all $x,y\in\set$.

$T:\set\to\E$ is called \emph{firmly nonexpansive}, if
\begin{align}
\norm{Tx-Ty}^2+\norm{(\id -T)x-(\id -T)y}^2\leq\norm{x-y}^2\label{eq:firmlynonexpansive}
\end{align}
holds for all $x,y\in\set$.
\end{mydef}
\begin{mylemma}[Proposition 4.2 \cite{BauCom}]
\label{lemma:firmlynonexpansive}
 Let $\set\subset \E$ be nonempty and let $T:\set\to\E$.
The following are equivalent
\begin{enumerate}[{(i)}]
\item $T$ is firmly nonexpansive on $\set$ 
\item $2 T- \id $ is nonexpansive on $\set$
\item $\norm{Tx-Ty}^2\leq\langle Tx-Ty,x-y\rangle$ for all $x,y\in \set$
\end{enumerate}
\end{mylemma}
%
%
%

\begin{mythm}[best approximation property - convex case]\label{t:best app}
 Let $C\subset\E$ be nonempty and convex, $x\in\E$ and $\point\in C$. $\point$ is the best approximation point $\point=P_C(x)$ if and only if
\begin{align}
\langle x-\point,y-\point \rangle\leq0\quad\fora y\in C. \label{eq:charBA}
\end{align}
If $C$ is a affine subspace then \eqref{eq:charBA} holds with equality.
\end{mythm}
\begin{proof}
For \eqref{eq:charBA} see Theorem 3.14 of \cite{BauCom},  while equality follows from Corollary 3.20 of the 
same.
\end{proof}


\begin{mythm}[(firm) nonexpansiveness of pro\-jec\-tors/\-re\-flec\-tors]\label{P:fn}
 Let $C$ be a closed, nonempty and convex set. The projector $P_C:\E\to\E$ is a firmly nonexpansive mapping and 
hence the reflector $R_C$ is nonexpansive.  If, in addition, $C$ is an affine subspace
then following conditions hold.
 \begin{enumerate}[{(i)}]
  \item\label{P:fn1}
  $P_C$ is firmly nonexpansive with equality, i.e.
\begin{align}
\norm{P_C x-P_C y}^2 +\norm{(\id -P_C)x-(\id -P_C)y}^2&=\norm{x-y}^2, \label{eq:P:fnonsubspace}
\end{align}
for all $x\in\E$.
\item \label{P:fn2}
For all $x\in\E$
\begin{align}
\norm{R_C  x-c}=\norm{x-c}\label{eq:Rdist}
\end{align}
holds for all $c\in C$.
 \end{enumerate} 
\end{mythm}
\begin{proof}
For the first part of the statement see \cite[Theorems 4.1 and 5.5]{Deutsch01}, \cite[Chapter 12]{GoebelKirk}, 
\cite[Propositions 3.5 and 11.2]{GoebelReich} and  \cite[Lemma 1.1]{Zarantonello}. 
The well-known refinement for affine subspaces follows by a routine application of the definitions and Theorem \ref{t:best app}.  
\end{proof}
%
%
%

\section{$(S,\varepsilon)$-firm nonexpansiveness}\label{s:S-epsilon-firm}
Up to this point, the results have concerned only convex sets, and hence the projector and related algorithms have all been single-valued.
In what follows, we generalize to nonconvex sets and therefore allow multi-valuedness of the projectors.
\begin{mylemma}\label{lemma:AAR}
 Let $A,B\subset\E$ be nonempty and closed. Let $x\in \E$. For any element  $x_+\in T_{DR}x$ 
there is a point $\tilde x\in R_AR_B x$ such that $x_+=\tfrac12(\tilde x +x)$.  Moreover, $T_{DR}$ 
satisfies the following properties. 
\begin{enumerate}[{(i)}]
 \item\label{lemma:AAR1}
 \begin{equation}\label{eq:AAR1}
\norm{x_+ - y_+}^2+\norm{\left(x - x_+\right)-\left(y -y_+\right)}^2
=\frac 1 2\norm{x-y}^2+\frac 1 2 \norm{\tilde x- \tilde y}^2
\end{equation}
where $x$ and $y$ are elements of $\E$, $x_+$ and $y_+$ are elements of $T_{DR}x$ and 
$T_{DR}y$ respectively, and $\tilde x\in R_AR_Bx $ and $\tilde y\in  R_AR_By$ are the corresponding 
points satisfying $x_+=\tfrac12(\tilde x + x)$ and $y_+=\tfrac12(\tilde y + y)$.
\item\label{lemma:AAR2}
For all $x\in \E$
\begin{align}
T_{DR}x&=\left\{ P_A (2z-x)-z +x~\middle|~ z\in P_B x\right\}.\label{eq:AAR2}
\end{align}
 \end{enumerate}
\end{mylemma}
\begin{proof}
By Definition \ref{AAR} 
\begin{align} 
x_+\in~ &T_{DR}x\\
\Longleftrightarrow~\,~\,~\qquad x_+\in~&\tfrac12(R_AR_Bx+x)\\
\Longleftrightarrow\quad 2 x_+-x\in~ &R_A R_Bx.
\end{align} 
Defining $\tilde x=2x_+-x$ yields $x_+=\tfrac12(\tilde x+x)$, where $\tilde x\in R_AR_Bx$.
\begin{enumerate}[{(i)}]
 \item 
For $x_+\in T_{DR}x$ (respectively $y_+\in T_{DR}y$) choose $\tilde x\in R_A R_B x$ (respectively $\tilde y$) such that $x_+=(\tilde x+x)/2$ (respectively $y_+$).
Then
\begin{align} 
&\norm{x_+ - y_+}^2+\norm{\left(x - x_+\right)-\left(y -y_+\right)}^2\\
&\qquad=\norm{\frac 1 2 \tilde x+ \frac 1 2 x -\frac 1 2 \tilde y -\frac 1 2 y}^2+
\norm{ \frac 1 2 x-\frac 1 2 \tilde x-\frac 1 2 y +\frac 1 2 \tilde y }^2 \\
&\qquad=\frac 1 2\norm{x-y}^2+\frac 1 2 \norm{\tilde x- \tilde y}^2
+\frac 1 2\langle \tilde x - \tilde y, x-y\rangle-\frac 1 2 \langle \tilde x - \tilde y, x-y\rangle\\
&\qquad=\frac 1 2\norm{x-y}^2+\frac 1 2 \norm{ \tilde x- \tilde y}^2.
\end{align} 
\item This follows easily from the definitions.  Indeed, we 
represent $v\in R_Bx$ as $v=2z-x$ for $z\in P_B x$ so that 
\begin{align} 
 T_{DR}x&=\left\{\frac 1 2\left(R_A v +x\right)~\middle|~ v\in R_B x\right\} \\
&=\left\{\frac 1 2\left(R_A (2z-x) +x\right)~\middle|~ z\in P_B x\right\},\\
&=\left\{\frac 1 2\left(2P_A (2z-x)-(2z-x) +x\right)~\middle|~ z\in P_B x\right\}\\
&=\left\{ P_A (2z-x)-z +x~\middle|~ z\in P_B x\right\}.
\end{align} 
\end{enumerate}

\end{proof}
\begin{myremark}
In the case where  $A$ and $B$ are convex, then 
as a consequence of Lemma \ref{lemma:AAR} i) and the fact that the reflector $R_\set$ onto a convex set $\set$ is nonexpansive,
we recover the well-known fact that $T_{DR}$ is firmly nonexpansive and
\eqref{eq:AAR1} reduces to
 \begin{align}
&\norm{T_{DR}x-T_{DR}y}^2+\norm{\left(\id -T_{DR}\right)x-\left(\id -T_{DR}\right)y}^2\notag\\
&\qquad\qquad=\frac 1 2\norm{x-y}^2+\frac 1 2 \norm{R_A R_B x- R_A R_B y}^2,\label{eq:AAR1_cvx}
\end{align}
while \eqref{eq:AAR2} reduces to
\begin{align}
 T_{DR}x&=x+P_A R_B x - P_B x .\label{eq:AAR2_cvx}
\end{align}
\hfill$\Box$
\end{myremark}


We define next an analog to firm nonexpansiveness in the nonconvex case with respect to a set $S$.
\begin{mydef}[$(S,\epsilon)$-(firmly-)nonexpansive mappings]
Let $D$ and $S$ be nonempty subsets of $\E$ and let $T$ be a (multi-\-valued) mapping from $D$ to $\E$.
\begin{enumerate}
   \item[i)]  $T$ is called \emph{$(S,\varepsilon)$-nonexpansive on $D$} if 
\begin{align}\label{eq:epsqnonexp}
\begin{aligned}
&\norm{x_+-\point_+}\leq\sqrt{1+\varepsilon}\norm{x-\point}\\
&\forall x\in D,~\forall \point\in S,~\forall x_+\in Tx,~\forall \point_+\in T\point. 
\end{aligned}
\end{align}
If \ref{eq:epsqnonexp} holds with $\epsilon=0$ then we say that $T$ is \emph{$S$-nonexpansive} on $D$.
\item[ii)] $T$ is called \emph{$(S,\varepsilon)$-firmly nonexpansive} on $D$ if  
\begin{align}\label{eq:quasifirm}
\begin{aligned}
&\norm{x_+-\point_+}^2+\norm{(x -x_+)-(\point -\point_+)}^2\leq(1+\varepsilon)\norm{x-\point}^2\\
&\forall x\in D,~\forall \point\in S,~\forall x_+\in Tx,~\forall \point_+\in T\point.
\end{aligned}
\end{align}
If \ref{eq:quasifirm} holds with $\epsilon=0$ then we say that $T$ is \emph{$S$-firmly nonexpansive} on $D$.

\end{enumerate}
\end{mydef}

Note that, as with (firmly) nonexpansive mappings, the mapping $T$ need not be a self-mapping from $D$ to itself.
In the special case where $S=\Fix T$, mappings satisfying \eqref{eq:epsqnonexp} are also 
called {\em quasi-(firmly-)nonexpansive}  \cite{BauCom}.  Quasi-nonexpansiveness is a restriction of
another well-known concept,  Fej\'er monotonicity, to $\Fix T$.
Equation \eqref{eq:quasifirm} is a relaxed version of firm nonexpansiveness \eqref{eq:firmlynonexpansive}.
The aim of this work is to expand the theory for projection methods (and in particular MAP and Douglas-Rachford) 
to the setting where one (or more) of the sets are nonconvex.
The classical (firmly) nonexpansive operator on $D$ is $(D,0)$-(firmly) nonexpansive on $D$. 

Analogous to the relation between firmly nonexpansive mappings and averaged mappings 
(see \cite[Chapter 4]{BauCom} and references therein) we have the 
following relationship between $(S,\varepsilon)$-firmly nonexpansive mappings and 
their $1/2$-averaged companion mapping.  
\begin{mylemma}[$1/2-$averaged mappings]
\label{lemma:averaged}
Let $D, S\subset\E$ be nonempty and $T:D\rightrightarrows\E$. The following are equivalent
\begin{enumerate}[{(i)}]
 \item \label{lemma:averaged1}
 $T$ is $(S,\varepsilon)$-firmly nonexpansive on $D$. 
 \item \label{lemma:averaged2}
The mapping $\widetilde T:D\rightrightarrows\E$ given by
\begin{align}
\widetilde Tx:=(2Tx-x)\quad \forall x\in D 
\end{align}
 is $(S,2\varepsilon)$-nonexpansive on $D$, i.e. $T$ can be written as
 \begin{align}\label{eq:charquasifirm}
 T x =\frac 1 2\left(x +\widetilde Tx\right)\quad \forall x\in D.
\end{align}

\end{enumerate}
\end{mylemma}
\begin{proof}
For $x\in D$ choose $x_+\in Tx$.
Observe that, by the definition of $\widetilde{T}$,  there is a corresponding $\tilde x\in \widetilde Tx$ such that
$x_+=\frac 1 2(x+\tilde x)$, which is just formula \eqref{eq:charquasifirm}.
Let $z$ be any point in $S$ and select any $z_+\in Tz$.  Then 
\begin{align} 
& \norm{x_+ - z_+}^2+\norm{ x- x_+ -(z-z_+)}^2\\
&\qquad\qquad=\norm{\tfrac12(x+\tilde x)-\tfrac12(z+\tilde z)}^2 + \norm{\tfrac12(x-\tilde x)-\tfrac12(z-\tilde z )}^2\\
&\qquad\qquad=\tfrac14\left[\norm{x-z}^2+2\langle x-z,\tilde x-\tilde z\rangle +\norm{\tilde x-\tilde z}^2\right]\\
&\qquad\qquad\qquad+\tfrac14\left[\norm{x- z}^2-2\langle x- z,\tilde x-\tilde z\rangle +\norm{\tilde x-\tilde z}^2\right]\\
&\qquad\qquad=\tfrac 12 \norm{x- z}^2+\tfrac 12\norm{\tilde x-\tilde  z}^2\\
&\qquad\qquad\stackrel{!}\leq\tfrac12 \norm{x- z}^2+\tfrac12(1+2\varepsilon)\norm{x- z}^2\\
&\qquad\qquad=(1+\varepsilon)\norm{x- z}^2 
\end{align} 
where the inequality holds if and only if $\widetilde T$ is $(S, 2\varepsilon)-$nonexpansive.  By definition, it then 
holds that $T$ is $(S,\epsilon)$-firmly nonexpansive if and only if $\widetilde T$ is $(S, 2\varepsilon)-$nonexpansive, 
as claimed.
\end{proof}
We state the following theorem to suggest that the framework presented in this work 
can be extended to a more general setting, for example the adaptive framework discussed in \cite{BauBorSIREV}.
It shows that the $(S,\varepsilon)$-firm nonexpansiveness is preserved under convex combination of operators.

\begin{mythm}
\label{t:combination qfirm}
Let $T_1$ be ($S,\varepsilon_1$)-firmly nonexpansive and 
 $T_2$ be ($S,\varepsilon_2$)-firmly nonexpansive on $D$.
 The convex combination $\lambda T_1+(1-\lambda)T_2$ is ($S,\varepsilon$)-firmly nonexpansive on $D$ where $\varepsilon=\max\{\varepsilon_1,\varepsilon_2\}$. 
\end{mythm}

\begin{proof}
Let $x,y\in D$. 
Let 
\begin{align} 
 x_+ \in& \lambda T_1x+\left(1-\lambda\right)T_2 x,\quad\textup{and}\\
 y_+ \in& \lambda T_1y+\left(1-\lambda\right)T_2 y,\\
\Rightarrow \quad x_+=& \lambda  x_+^{(1)} +(1-\lambda) x_+^{(2)},\quad\textup{where } x_+^{(1)}\in T_1 x,~ x_+^{(2)}\in T_2x\\ 
 y_+=& \lambda  y_+^{(1)} +(1-\lambda) y_+^{(2)},\quad\textup{where } y_+^{(1)}\in T_1 y,~ y_+^{(2)}\in T_2y.
\end{align} 
By Lemma \ref{lemma:averaged} (\ref{lemma:averaged2})  one has nonexpansiveness of the mappings given by $2T_1x-x$ and $2T_2x-x$, $x\in D$ that is
\begin{align} 
\norm{\left[ 2 x_+^{(1)}-x\right]-\left[2 y_+^{(1)}-y\right] }\leq& \sqrt{1+2\varepsilon_1}\norm{x-y},\\
\norm{\left[ 2 x_+^{(2)}-x\right]-\left[2 y_+^{(2)}-y\right] }\leq& \sqrt{1+2\varepsilon_2}\norm{x-y}.
\end{align} 
This implies
\begin{align} 
 &\norm{\left(2 x_+-x\right)-\left(2 y_+-y\right)}\\
=&\norm{\left(2\left[\lambda  x_+^{(1)} +(1-\lambda) x_+^{(2)}\right]-x\right)-\left(2\left[ \lambda y_+^{(1)} +(1-\lambda) y_+^{(2)}\right]-y\right)}\\
=&\norm{\lambda\left(\left[2  x_+^{(1)} -x\right]-\left[2  y_+^{(1)} -y\right]\right)-(1-\lambda)\left(\left[2  x_+^{(2)} -x\right]-\left[2  y_+^{(2)} -y\right]\right)}\\
\leq&\lambda\norm{\left[2  x_+^{(1)} -x\right]-\left[2  y_+^{(1)} -y\right]}+(1-\lambda)\norm{\left[2  x_+^{(2)} -x\right]-\left[2  y_+^{(2)} -y\right]}\\
\leq&\sqrt{1+2\varepsilon}\norm{x-y}.
\end{align} 
Now using Lemma \ref{lemma:averaged} (\ref{lemma:averaged1}) the proof is complete.
\end{proof}

\subsection{Regularity of Sets}
To assure property \eqref{eq:quasifirm} for the projector and the Douglas-Rachford operator, we 
determine the inheritance of the regularity of the projector and reflectors from the regularity of 
the sets  $A$ and $B$ upon which we project.  We begin with some established notions of 
set regularity and introduce a new, weaker form that will be central to our analysis. 
\begin{mydef}[Prox-regularity]\label{proxregularity}
A nonempty (locally) closed set $\set\subset\E$ is \emph{prox-regular} at a point $\point\in\set$ if the projector $P_\set$ is 
single-valued around $\point$. 
\end{mydef}
What we take as the definition of prox-regularity 
actually follows from the equivalence of prox-regularity of sets as defined in 
\cite[Definition 1.1]{PolRockThib00} and the single-valuedness of the projection
operator on neighborhoods of the set \cite[Theorem 1.3]{PolRockThib00}.

\begin{mydef}[normal cones]
The 
{\em proximal normal cone} $N^P_\set(\point)$ to a set $\set\subset\E$ at a point $\point\in\set$ is
defined by 
\begin{align}
N^P_\set(\point)\equiv \cone(P^{-1}_\set(\point)-\point). \label{eq:pnormal}
\end{align}
The the \emph{limiting normal cone} $ N_\set(\point)$ is defined as any vector that can be written as the limit 
of proximal normals; that is, $\overline v\in N_\set(\point)$ if and only if there exist sequences $(x_k)_{k\in\mathbb{N}}$ in $\set$ and 
$(v_k)_{k\in \mathbb{N}}$ in $N^P_\set(x_k)$ such that $x_k\to \point$ and $v_k\to \overline v$. 
\end{mydef}
The construction of the \emph{limiting normal cone} goes back to Mordukhovich (see \cite[Chap. 6 Commentary]{VA}).
\begin{myprop}[Mordukhovich]
The \emph{limiting normal cone} or \emph{Mordukhovich normal cone} is the smallest cone satisfying the two properties
\begin{enumerate}
 \item  $P^{-1}_\Omega(\point)\subseteq (I + N_\Omega)(\point)$ where $P_\Omega^{-1}(\point)$ is the preimage set of $\point$ under $P_\Omega$.
 \item for any sequence $x_i\to\point$ in $\set$ any limit of a sequence of normals $v_i\in N_\set(x_i)$ must lie in $ N_\set(\point)$.
\end{enumerate}
\end{myprop}

\begin{mydef}[{$(\varepsilon,\delta)$-(sub)regularity}]\label{epsilondeltasubregular}
$~$
\begin{enumerate}
   \item[i)] A nonempty set $\set\subset\E$ is \emph{($\varepsilon,\delta$)-subregular} at $\hat x$ with respect to $S\subset\E$ if
\begin{align}\label{eq:epsilondeltasubregularity}
\langle v_x,\point-x\rangle\leq\varepsilon\norm{v_x}\norm{\point-x}
\end{align}
holds for all $x\in   \Ball_\delta(\hat x)\cap\set,$ $\point\in S\cap   \Ball_\delta(\hat x),$ $v_x\in  N^P_\set(x)$. 
We simply say $\set$ is ($\varepsilon,\delta$)-subregular at $\hat x$ if $S=\{\hat x\}$.
\item[ii)]  If $S=\set$ in i) then we say that the set $\set$ is \emph{$(\varepsilon,\delta)$-regular} at $\hat x$.
\item[iii)] If for all $\epsilon>0$ there exists a $\delta>0$ such that \eqref{eq:epsilondeltasubregularity} holds 
for all $x,\point\in   \Ball_\delta(\point)\cap\set$ and $v_x\in  N_\set(x)$, then $\set$ is said to be 
{\em super-regular}. 
\end{enumerate}
\end{mydef}
The definition of ($\varepsilon,\delta$)-regularity was introduced in \cite[Definition 9.1]{BLPW1} and is a 
generalization of the notion of super-regularity introduced in \cite[Definition 4.3]{LLM}.
More details to $(\varepsilon,\delta)$-regularity can be seen in \cite{BLPW1}.  Of particular interest is the following
proposition.  Preparatory to this, we remind readers of another well-known type of regularity,  {\em Clarke regularity}. 
To avoid introducing the F\'echet normal which is conventionally used to define Clarke regularity, we follow  \cite{LLM} 
which uses proximal normals. 
\begin{mydef}[Clarke regularity]
 A nonempty (locally) closed set $\set\subset\E$ is \emph{Clarke regular} at a point $\point\in\set$ if, for all 
$\varepsilon>0$, any two points $u, ~z$ close enough to $\point$ with $z\in \set$, and any point $y\in P_\set(u)$, 
satisfy $\left\langle z-\point, u-y \right\rangle\leq \varepsilon\|z-\point\|\|u-y\|$. 
\end{mydef}

\begin{myprop}[Prox-regular implies super-regular,  Proposition 4.9 of \cite{LLM}]\label{def:clarke}
 If a closed set $\set\subset\E$ is prox-regular at a point in $\set$, then it is super-regular at that point.
 If a closed set $\set\subset\E$ is super-regular at a point in $\set$, then it is Clarke regular at that point.
 \end{myprop}
Super-regularity is something between Clarke regularity and amenability or prox\-reg\-ul\-ar\-ity.   
 $(\epsilon,\delta)$\--regular\-ity is weaker still than Clarke regularity (and hence super-regularity)
as the next example shows. 

\begin{myremark}
 $(\epsilon,\delta)$-regularity does not imply Clarke regularity
\begin{proof}
 \begin{equation}
  \set:=\left\{(x_1,x_2)\in\R^2~\middle|~
  \begin{matrix}
   x_2\leq -x_1&\textup{if }~ x_1\leq 0\\ x_2\leq0&\textup{if }~x_1>0
  \end{matrix}
\right\}.
 \end{equation}
For any $x_-\in \partial_-\set:=\left\{(x_1,x_2)~|~ x_2=-x_1,~x_1< 0\right\}$ and
$x_+\in \partial_+\set:=\left\{(x_1,0)~|~ x_1> 0\right\}$ one has
\begin{align} 
N_\set(x_-)=&\left\{(\lambda,\lambda)~\middle|~ \lambda\in \R^+\right\}\\
N_\set(x_+)=&\left\{(0,\lambda)~|~ \lambda\in \R^+\right\}
\end{align} 
which implies $N_\set(0)=N_\set(x_-)\cup N_\set(x_+)$. Since $N^P_\set(0)=\left\{0\right\}$ the set $\set$ is not Clarke regular at $0$.
Define $\nu_-=(\sqrt{2}/2,\sqrt{2}/2)\in N_\set(x_-)$, $\nu_+=(0,1)\in N_\set(x_+)$ and note 
\begin{align} 
\langle \nu_-,0-x_-\rangle=0\quad\textup{and}\quad
\langle \nu_+,0-x_+\rangle=0
\end{align} 
to show that $\set$ is  $(0,\infty)-$subregular at $0$.
By the use of these two inequalities one now has
\begin{align} 
\langle \nu_-,x_+-x_-\rangle=\langle \nu_-,x_+\rangle\leq\sqrt{2}/2\norm{x_+}\\
\langle \nu_+,x_- - x_+\rangle=\langle \nu_+,x_-\rangle\leq\sqrt{2}/2\norm{x_-}
\end{align} 
and hence $\set$ is $(\tfrac{\sqrt{2}}2,\infty)$-regular.
\end{proof}
\end{myremark}

\begin{myremark}[Example \ref{examples} \ref{ex:cross}) revisited]\label{remark:clarke}
The set 
\begin{equation}
  A:=\R\times \{0\}\cup\{0\}\times\R
 \end{equation}
is a particularly easy pathological set that illustrates the distinction between our new notion of 
subregularity and previous notions found in the literature.
Note that for $x_1\in \R\times \{0\}$, $N_A(x_1)=N_A^P(x_1)= \{0\}\times\R$ and for
for $x_2\in \{0\}\times\R$, $N_A(x_2)=N_A^P(x_2)= \R\times\{0\}$ and that $N_A(0)=A$ and $N_A^P(0)=0$ which 
implies that at the origin $A$ is not Clarke regular and therefore neither super-regular nor prox-regular there.  In fact, it is not 
even $(\epsilon,\delta)$-regular at the origin for any $\epsilon<1$ and any $\delta>0$.
The set $A$ is, however, $(0,\infty)$-\emph{sub}regular at $\{0\}$. Indeed, for any $x_1\in \R\times \{0\}$ one has 
$\nu_1\in N_A(x_1)= \{0\}\times\R$ and therefore $\langle \nu_1,x_1-0\rangle=0$. 
Analogously for $x_2\in \{0\}\times\R$, $\nu_2\in N_A(x_2)= \R\times\{0\}$ and it follows that $\langle \nu_2,x_2-0\rangle=0$, 
which shows that $A$ is $(0,\infty)$-subregular at $0$.
\hfill$\Box$
\end{myremark}

\subsection{Projectors and Reflectors}
We show in this section how $(S,\varepsilon)$(firm)-nonexpansiveness of projectors and reflectors is 
a consequence of (sub)regularity of the underlying sets. 

\begin{mythm}[projectors and reflectors onto $(\varepsilon,\delta$)-subregular sets]
\label{t:subreg proj-ref}
Let $\set\subset\E$ be nonempty closed and $(\varepsilon,\delta$)\--sub\-reg\-ul\-ar at $\hat x$ 
with respect to $S\subseteq \Omega\cap \Ball_\delta(\hat x)$ and define
\begin{align}\label{eq:UPsubreg}
U\equiv\left\{ x\in \E ~\middle|~ P_\set x\subset \Ball_\delta(\hat x)\right\}. 
\end{align}
\begin{enumerate}[{(i)}]
\item \label{t:subreg proj-ref1}
The projector $P_\set$ is $(S,\relconst_1)$-nonexpansive on $U$, that is,\\
($\forall x\in  U$) ($\forall x_+\in P_\set x$) ($\forall \point\in S$) 
\begin{align}\label{eq:P:delta}
\begin{aligned}
\norm{x_+ -\point}\leq \sqrt{1+\relconst_1}\norm{x-\point}
\end{aligned}
\end{align}
where $\relconst_1\equiv 2\varepsilon+\varepsilon^2$.

   \item \label{t:subreg proj-ref2} 
The projector $P_\set$ is $(S,\relconst_2)$-firmly nonexpansive on $ \Ball_\delta(\hat x)$, that is,
($\forall x\in  U$) ($\forall x_+\in P_\set x$) ($\forall \point\in S$)
\begin{align}\label{eq:P:fna}
\begin{aligned}
\norm{x_+ -\point}^2+\norm{x-x_+ }^2\leq (1+\relconst_2)\norm{x-\point}^2, 
\end{aligned}
\end{align}
where $\relconst_2\equiv2\varepsilon+2\varepsilon^2.$
\item \label{t:subreg proj-ref3} 
The reflector $R_\set$ 
is $(S,\relconst_3)$-nonexpansive on $ \Ball_\delta(\hat x)$, that is,\\
($\forall x\in  U$) ($\forall x_+\in R_\set x$) ($\forall \point\in S$)
\begin{align}\label{eq:R:nonexa}
\begin{aligned}
\norm{x_+-\point}\leq\sqrt{1+\relconst_3} \norm{x-\point}, 
\end{aligned}
\end{align}
where $\relconst_3\equiv 4\varepsilon+4\varepsilon^2$.
\end{enumerate}
 \end{mythm}
\begin{proof}
\begin{enumerate}[{(i)}]
\item  The projector is nonempty since $\Omega$ is closed.  Then  by the 
Cauchy-Schwarz inequality
 \begin{align} 
\norm{x_+ -\point}^2
=& \langle x-\point,x_+ -\point\rangle+\langle x_+ -x,x_+ -\point\rangle\\
\leq& \norm{x-\point}\norm{x_+ -\point}+\langle x_+ -x,x_+ -\point\rangle.\label{e:inter1}\\
\intertext{Now for $x\in U$ we have also that $x_+ \in \Ball_\delta(\hat x)$ and thus, by 
the  definition of $(\varepsilon,\delta)$-subregularity with respect to $S$, 
($\forall x\in  U$)~($\forall  x_+ \in P_\set x$)~($\forall \point\in S$)   
}
\langle x_+ -x,x_+ -\point\rangle\leq& \varepsilon \norm{x-x_+ }\norm{x_+ -\point}\\
\leq &\varepsilon\norm{x-\point}\norm{x_+ -\point}.
\intertext{Combining this with \eqref{e:inter1}  yields ($\forall x\in  U$)~($\forall  x_+ \in P_\set x$)~($\forall \point\in S$)   }
\norm{x_+ -\point}\leq& (1+\varepsilon)\norm{x-\point}\\ 
=&\sqrt{1+(2\varepsilon+\varepsilon^2)}\norm{x-\point}\\
\end{align} 
 as claimed. 
\item  Expanding and rearranging the norm yields
($\forall x\in  U$)~($\forall  x_+ \in P_\set x$)~($\forall \point\in S$)   
\begin{align}
&\!\!\!\!\!\!\norm{x_+ -\point}^2+\norm{x-x_+ }^2\notag\\
&=~\norm{x_+ -\point}^2+\norm{x-\point+\point -x_+ }^2\notag\\
&=~\norm{x_+ -\point}^2+\norm{x-\point}^2+2\langle x-\point,\point-x_+ \rangle+\norm{x_+ -\point}^2\notag\\
&=2\norm{x_+ -\point}^2+\norm{x-\point}^2+2\underbrace{\langle x_+ -\point,\point-x_+ \rangle}_{=-\norm{x_+ -\point}^2}
+2\underbrace{\langle x-x_+ ,\point-x_+  \rangle}_{\leq\varepsilon\norm{x-x_+ }\norm{x_+ -\point}}\notag\\
&\leq~ \norm{x-\point}^2+2\varepsilon\norm{x_+ -\point}\norm{x-x_+ }\label{eq:P:prf1}
\end{align}
where the last inequality follows from the definition of $(\varepsilon,\delta$)-subregularity 
with respect to $S$. 
By definition, $\norm{x-x_+ }=\dist {x} {\set}\leq \norm{x-\point}$. 
 Combining \eqref{eq:P:prf1} and equation \eqref{eq:P:delta} yields 
($\forall x\in  U$)~($\forall  x_+ \in P_\set x$)~($\forall \point\in S$)   
\begin{align} 
\norm{x_+ -\point}^2+\norm{x-x_+ }^2
\leq
&\left(1+2\varepsilon\left(1+\varepsilon\right)\right)\norm{x-\point}^2.
\end{align} 
\item By (\ref{t:subreg proj-ref2}) the projector is $(S,2\varepsilon +2\varepsilon^2)$-firmly nonexpansive on $U$,
and so by Lemma \ref{lemma:averaged} (\ref{lemma:averaged2}) $R_\set=2 P_\set -\textup{Id}$ is 
$(S,4\varepsilon +4\varepsilon^2)$-nonexpansive on $U$.
\end{enumerate}
This completes the proof.
\end{proof}
Note that $\relconst_1<\relconst_2$ ($\varepsilon>0$) in the above theorem, in other words, the 
{\em degree} to which classical firm nonexpansiveness is violated is greater than the degree to which 
classical nonexpansiveness is violated.  This is as one would expect since firm nonexpansiveness is a 
stronger property than nonexpansiveness.   

%

We can now characterize the degree to which the Douglas-Rachford operator violates firm\--non\-expansive\-ness 
on neighborhoods of  $(\varepsilon,\delta)$-subregular sets.  
\begin{mythm}[$(S,\tilde\varepsilon)$-firm nonexpansiveness of $T_{DR}$]\label{t:Velvet Underground}
 Let $A,B\subset\E$ be closed and nonempty. Let $A$ and $B$ be $(\varepsilon_A,\delta)$- and 
$(\varepsilon_B,\delta)$\--sub\-regular respectively  at $\hat x$ with respect to $S\subset\Ball_\delta(\hat x)\cap \left( A\cap B\right)$.
Let $T_{DR}:\E\rightrightarrows\E$ be the Douglas-Rachford operator defined by \eqref{eq:AAR} and define
\begin{align}\label{eq:UAARsubreg}
U\equiv\left\{ z\in \E~|~P_Bz\subset \Ball_\delta(\hat x)\mbox{ and }P_AR_Bz\subset \Ball_\delta(\hat x)\right\}.
\end{align}
Then $T_{DR}$ is $(S,\tilde\varepsilon)$-firmly nonexpansive on $ U$
where
\begin{equation}\label{eq:eps DR}
\relconst=2\varepsilon_A(1+\varepsilon_A)+2\varepsilon_B(1+\varepsilon_B)+8\varepsilon_A(1+\varepsilon_A)\varepsilon_B(1+\varepsilon_B).
\end{equation}
That is, ($\forall x\in  U$)~($\forall  x_+ \in T_{DR} x$)~($\forall \point\in S$)   
\begin{align}
\begin{aligned}\label{eq:DR:fna}
\norm{x_+ -\point}^2+\norm{x -x_+ }^2\leq&\left(1+\relconst\right)\norm{x-\point}^2.
\end{aligned}
\end{align}
\end{mythm}
\begin{proof}
Define  
$U_A\equiv\{z~|~P_Az\subset \Ball_\delta(\hat x)\}$. By Theorem 
\ref{t:subreg proj-ref}(\ref{t:subreg proj-ref3}) 
($\forall y\in  U_A$)~($\forall  \tilde x \in R_A y$)~($\forall \point\in S$) 
\begin{align} 
 \norm{\tilde x-\point}\leq&\sqrt{1+4\varepsilon_A(1+\varepsilon_A)}\norm{y-\point}.\label{e:TDR1}\\
\intertext{Similarly, define $U_B\equiv\{z~|~P_Bz\subset \Ball_\delta(\hat x)\}$ and again apply 
Theorem \ref{t:subreg proj-ref}(\ref{t:subreg proj-ref3}) to get 
($\forall x\in  U_B$)~($\forall  y \in R_B x$)~($\forall \point\in S$) 
 } 
\norm{y-\point}\leq&\sqrt{1+4\varepsilon_B(1+\varepsilon_B)}\norm{x-\point}.\label{e:TDR2}
\intertext{Now, we choose any $x\in U_B$ such that $R_Bx\in U_A$, that is 
$x\in U$,
so that we can combine  \eqref{e:TDR1}-\ref{e:TDR2} to get 
($\forall x\in  U$)~($\forall  \tilde x \in R_A R_B x$)~($\forall \point\in S$) 
}
 \norm{\tilde x-\point}\leq&\sqrt{1+4\varepsilon_A(1+\varepsilon_A)}\sqrt{1+4\varepsilon_B(1+\varepsilon_B)}
\norm{x-\point} = 
\sqrt{1+2\relconst}\norm{x-\point}.\label{e:TDR3}
\end{align} 
Note that $R_AR_B\point=R_B\point=\point$ since $\point\in A\cap B$, 
so \eqref{e:TDR3} says that the operator  $\widetilde T\equiv R_AR_B$
is $(S,\relconst)$-nonexpansive on $U$.  Hence by  
Lemma \ref{lemma:averaged} $T_{DR} = \tfrac12\left(\widetilde T + I \right)$ is 
$(S,2\relconst)$-firmly nonexpansive on $U$, as claimed.
\end{proof}

If one of the sets above is convex, say $B$ for instance, the constant $\relconst$ simplifies to $\relconst=2\varepsilon_A(1+\varepsilon_A)$ since
$B$ is $(0,\infty)$-subregular at $\point$. 


\section{Linear Convergence of Iterated $(S,\varepsilon)$-firmly nonexpansive Operators}\label{s:convergence}
Our goal in this section is to establish the weakest conditions we can (at the moment) 
under which the MAP and Douglas-Rachford algorithms
converge locally linearly.  The notions of regularity developed in the previous section 
are necessary, but not sufficient.  In addition to regularity of the operators, we need 
regularity of the fixed point sets of the operators. This is developed next.  

Despite its simplicity, the following Lemma is one of our fundamental tools.
\begin{mylemma}\label{Tconv}
Let $D\subset\E$, $S\subset\Fix T$, ${T}:{D}\rightrightarrows{\E}$ and $U\subset D$.
If 
\begin{enumerate}
   \item[(a)] $T$ is  $(S,\varepsilon)$-firmly nonexpansive on $U$ and 
\item[(b)] for some $\lambda>0$, T satisfies the  \emph{coercivity condition}
\begin{align}\label{eq:Tcoerciv}
\norm{x-x_+}\geq \lambda \,\dist{x}{S}\quad
\forall~x_+\in Tx, ~\forall x\in U.
\end{align}
\end{enumerate}
Then 
\begin{align}\label{eq:Tconv}
\dist{x_+}{S}\leq&\sqrt{(1+\varepsilon-\lambda^2)}~\dist{x}{S}
\quad
\forall~x_+\in Tx, ~\forall x\in U.
\end{align}
\end{mylemma}

\begin{proof}
For $x\in U$ choose any $x_+\in Tx$, and define $\point:= P_S x$.  
Combining equations \eqref{eq:Tcoerciv} and \eqref{eq:quasifirm} yields
  \begin{align} 
 \norm{x_+-\point}^2+\left(\lambda \norm{x-\point}\right)^2\stackrel{b)}\leq&\\
 \norm{x_+-\point}^2+\norm{x -x_+}^2\stackrel{{a)}}\leq&{(1+\varepsilon)}\norm{x-\point}^2,
 \end{align} 
which immediately yields 
  \begin{align}\label{eq:Tconfproof1}
 \norm{x_+-\point}^2\leq&{(1+\varepsilon-\lambda^2)}\norm{x-\point}^2.
 \end{align}
Since $\point\in S$ by definition one has $\dist{x_+}{S }\leq \norm{x_+-\point}$. Inserting this in \eqref{eq:Tconfproof1} 
and using the fact $\norm{x-\point}=\dist x S$ then proves \eqref{eq:Tconv}.
\end{proof}


\subsection{Regularity of Intersections of Collections of Sets}
To this point, we have shown how the regularity of sets translates to the 
degree of violation of (firm) nonexpansiveness of projection-based fixed point mappings.  
What remains is to develop sufficient conditions for guaranteeing \eqref{eq:Tcoerciv}.
For this we define a new notion of regularity of collections of sets which generalizes through localization two 
well-known concepts.  
The first concept, which we call \emph{strong regularity} of the collection, has many different names in the literature, among them 
{\em linear regularity} \cite{LLM}.  We will use the term {\em linear regularity} of the collection to denote the second key concept 
upon which we build.  Our generalization is called {\em local linear regularity}.  
Both terms ``strong'' and ``linear'' are overused in the literature but we have attempted, at 
the risk of some confusion, to conform to the usage that best indicates the heritage of the ideas. 

\begin{mydef}[strong regularity, Kruger \cite{Kruger2004}]\label{strongregular}
A collection of $m$ closed, nonempty sets $\set_1,\set_2,\dots,\set_m$ is \emph{strongly regular at $\point$}
if there exists an $\alpha>0$ and a $\delta>0$ such that
\begin{align}
\left(\cap_{i=1}^m (\set_i-\omega_i-a_i)\right)\cap \Ball_\rho\neq\emptyset\label{eq:strongregular}
\end{align}
for all $\rho\in(0,\delta]$, $\omega_i\in\set_i\cap\Ball_\delta(\point),$ $a_i\in B_{\alpha\rho}$, $i=1,2,\dots,m$.
\end{mydef}

\begin{mythm}[Theorem 1 \cite{Kruger2006}]
 A collection of closed, nonempty sets $\set_1$, $ \set_2$, $\dots$ ,$\set_m$ is \emph{strongly regular} 
at $\point$ if and only if there exists a $\kappa>0$ and a $\delta>0$ such that
\begin{align}
\dist{x}{\cap_{j=1}^m (\set_j-x_j)}\leq\kappa \max_{i=1,\dots,m}\dist{x+x_i}{\set_i},\quad\forall x\in  \Ball_\delta(\point), \label{eq:charstrongregular}
\end{align}
for all $x\in   \Ball_\delta(\point)$, $x_i\in  \Ball_\delta$, $i=1,\dots, m$.
\end{mythm}

\begin{mythm}[Theorem 1 \cite{Kruger2004}]
 A collection of closed sets $\set_1$, $\set_2$,$\dots$, $\set_m\subset\E$ 
is \emph{strongly regular} \eqref{eq:strongregular} at a point $\point\in\cap_i \set_i$, 
if the only solution to the system 
\begin{equation}
 \sum_{i=1}^m v_i=0,\qquad\textup{with } v_i\in N_{\set_i}(\point)\quad\for i=1,2,\dots,m
\end{equation}
is $v_i=0$ for $i=1,2,\dots,m$. 
For two sets $\set_1,\set_2\subset\E$ this can be written as
\begin{align}
 N_{\set_1}(\point)\cap -N_{\set_2} (\point)=\{0\},\label{linearregularcone}
\end{align}
and is equivalent to the previous Definition \eqref{eq:strongregular}.
\end{mythm}


\begin{mydef}[linear regularity]\label{localregular}
 A collection of closed, nonempty sets $\set_1$, $\set_2$, $\dots$, $\set_m$ is 
\emph{locally linearly regular} at $\hat x\in \cap_{j=1}^m \set_j$ on $\Ball_\delta(\hat x)$
($\delta>0$) if there exists a 
$\kappa>0$ such that, for all $x\in  \Ball_\delta(\hat x)$,
\begin{align}\label{eq:locallinear}
\dist{x}{\cap_{j=1}^m \set_j}\leq\kappa \max_{i=1,\dots,m}\dist{x}{\set_i}. 
\end{align}
The infimum over all $\kappa$ such that \eqref{eq:locallinear} holds is called \emph{regularity modulus}.
If there is a $\kappa>0$ such that \eqref{eq:locallinear} holds for all $\delta>0$ (that is, 
for all $x\in \E$) the collection of sets is 
called \emph{linearly regular} at $\hat x$.
\end{mydef}

\begin{myremark}\label{remark:stronglinear}
Since \eqref{eq:locallinear} is \eqref{eq:charstrongregular} with $x_j=0$ for all $j=1,2,\dots,m$, it is clear that 
strong regularity implies local linear regularity (for some $\delta>0$) and is indeed a much more restrictive notion than local linear regularity.
   What we are calling local linear regularity at $\hat x$ has appeared in various forms elsewhere.  See for 
instance \cite[Proposition 4]{Ioffe2000}, \cite[Section 3]{NgaiThera01}, and  \cite[Equation (15)]{Kruger2006}. 
 Compare this to (bounded) linear regularity defined in \cite[Definition 5.6]{BauBorSIREV}.
Also compare this to the {\em basic constraint qualification for sets} in \cite[Definition 3.2]{Mor06} and 
{\em strong regularity} of the collection in \cite[Proposition 2]{Kruger2006}, also called {\em linear regularity} in 
\cite{LLM}.
\hfill$\Box$\end{myremark}

\begin{myremark}\label{remark:overview}
 Based on strong regularity (more specifically, characterization \eqref{linearregularcone}) Lewis, Luke and Malick proved local linear convergence 
of MAP in the nonconvex setting, 
where both sets $A,B$ are closed and one of the sets is \emph{super-regular} \cite{LLM}. This was refined later in \cite{BLPW2}.
The proof of convergence that will be given in this work is different from the one used in \cite{LLM,BLPW2} and more related to the one 
in \cite{BauBorSIREV}.
Convergence is achieved using (local) linear regularity \eqref{eq:locallinear}, which is described in \cite[Theorem 4.5]{DeutschCPA03} as 
\emph{the precise property equivalent to uniform linear convergence of the CPA} (Cyclic projections algorithms).
However the rate of convergence achieved by the use of linear regularity is not optimal, while the one in \cite{LLM,BLPW2} is in some instances.
An adequate description of the relation between the \emph{direct/primal} techniques used here and the \emph{dual} approach 
used in \cite{LLM,BLPW2} is a topic of future research.
\hfill$\Box$\end{myremark}

\begin{mythm}[linear regularity of collections of convex cones]\label{t:Gruene Libanon}
 Let  $\set_1$, $\set_2$, $\dots$, $\set_m$ be  a collection of closed, nonempty, convex cones.
The following statements are equivalent
\begin{enumerate}[{(i)}]
 \item There is a $\delta>0$ such that the collection is locally linearly regular 
at $\hat x\in\cap_{j=1}^m\set_j$ on $\Ball_\delta(\hat x)$.
 \item The collection is linearly regular at $\hat x\in\cap_{j=1}^m\set_j$
\end{enumerate}
\begin{proof}
 \cite[Proposition 5.9]{BauBorSIREV}
\end{proof}
\end{mythm}
\begin{myex}[Example \ref{examples} revisited.]\label{ex:revisited}
The collection of sets in example \ref{examples} (\ref{ex:lines2}) is strongly regular at $0$ ($c=\sqrt 2/2$) and linearly regular ($\kappa=\sqrt 2 /4$).
The same collection of sets embedded in a higher-dimensional space is still linearly regular, but looses its strong regularity.
This can be seen by shifting one of the sets in example \ref{examples} (\ref{ex:lines3}) in $x_3$-direction, as this renders the 
intersection empty.  This shows that \emph{linear regularity does not imply strong regularity}.
The collection of sets in example (\ref{ex:lineball}) is neither strongly regular nor linearly regular.
The collection of sets in Example \ref{examples} (\ref{ex:cross})  is strongly regular at the intersection.
One has $N_B(0)=\left\{(\lambda,-\lambda)\middle|~\lambda\in\R\right\}$ and by Remark \ref{remark:clarke} $N_A(0)=A$ and this directly shows
$N_A(0)\cap-N_B(0)=\{0\}$.   
In example \ref{examples} (\ref{ex:Borwein}) one of the sets is nonconvex, 
but the collection of sets is still well-behaved in the sense that it is both strongly and linearly regular.
It is worth emphasizing, however, that the set $A$ in Example \ref{examples} (\ref{ex:cross}) is not Clarke regular at the origin.  
This illustrates the fact that collections of classically ``irregular'' sets can still be quite regular at points of intersection.  
\qed
\end{myex}

\subsection{Linear Convergence of MAP}
In the case of the MAP operator, the connection between local linear regularity of the collection 
of sets and the coercivity of the operator with respect to the intersection is natural, as the next result shows.
\begin{myprop}[coercivity of the projector]\label{t:proj coerciv}
 Let $ A,B $ be nonempty and closed subsets of $\E$, $\hat x\in S\equiv A\cap B$ and let the collection $\{A,B\}$ be \emph{locally linearly regular} 
at $\hat x$ on $\Ball_\delta(\hat x)$  with constant $\kappa$ for some $\delta>0$.
One has
 \begin{align}
 \begin{aligned}
  \norm{x -x_+}\geq& \gamma\, \dist x S\quad  \forall x_+\in P_B x,~\forall x\in A\cap \Ball_\delta(\hat x)
\end{aligned}
  \end{align}
 where $\gamma=1/\kappa$.
\end{myprop}

\begin{proof}
By the definition of the distance and the projector one has, for $x\in A$ and any $x_+\in P_B x$,
\begin{align} 
 \norm{x-x_+}=&\dist{x}{ B }\\
=&\max\left\{\dist{x}{ B },\underbrace{\dist{x}{ A}}_{=0}\right\}\\
\geq&\gamma\, \dist{x}{ S }.
\end{align} 
The inequality follows by Definition \ref{localregular} (local linear regularity at $\hat x$ on $\Ball_\delta(\hat x)$ with 
constant $\kappa$), since $x\in \Ball_\delta(\hat x)$.
\end{proof}

\begin{mythm}[Projections onto a ($\varepsilon,\delta$)-subregular set]\label{P:epsilondeltareg}
 Let $ A, B $ be nonempty and closed subsets of $\E$ and let $\hat x\in S \equiv  A\cap B $. 
If 
\begin{enumerate}
\item[(a)]   $ B $ is ($\varepsilon,\delta$)-subregular at $\hat x$ with respect to  $S$  and
   \item[(b)] the collection $\{A,B\}$ is locally linearly regular at $\hat x$ on $\Ball_\delta(\hat x)$
\end{enumerate}
then
\begin{align}\label{eq:AP:superregular}
\dist{x_{+}}{ S }\leq\sqrt{1+\relconst-\gamma^2}~\dist{x}{ S } ,\quad \forall x_+\in P_B x, ~\forall x\in U 
\end{align}
where $\gamma=1/\kappa$ with $\kappa$ the regularity modulus on $\Ball_\delta(\hat x)$, $\relconst=2\varepsilon+2\varepsilon^2$ and 
 \begin{align}\label{eq:AP:regionofconvergence}
 U\subset \left\{x\in A\cap \Ball_\delta(\hat x)~\middle|~P_B x\subset\Ball_\delta(\hat x) \right\}.
 \end{align}

\end{mythm}

\begin{proof}

Since $ B $ is ($\varepsilon,\delta$)-subregular at $\hat x$ with respect to  $S$ one can apply Theorem \ref{t:subreg proj-ref} 
to show that the projector $P_{ B }$ is $(S,2\varepsilon+2\varepsilon^2)$-firmly nonexpansive
on $U$.   Moreover, condition (b) and Proposition \ref{t:proj coerciv} yield  
\begin{align}
 \norm{x_+-x}\geq \gamma \dist {x}{S}\quad\forall x_+\in P_B x,~\forall x\in U.
\end{align}
Combining $(a)$ and $(b)$ and applying Lemma \ref{Tconv} then gives
 \begin{align} 
\dist{x_+}{S}\leq\sqrt{1+2\relconst-\gamma^2}\dist {x} {S},\quad \forall x_+\in P_B x, ~\forall x\in U .
\end{align} 
\end{proof}

\begin{mycor}[Projections onto a convex set \cite{GPR65} ]
\label{t:GPR65}
 Let $ A$ and $B$ be nonempty, closed subsets of $\E$. 
If 
\begin{enumerate}
   \item[(a)] the collection $\{A,B\}$ is locally linearly regular at $\hat x\in A\cap B$ on $\Ball_\delta(\hat x)$ 
with regularity modulus $\kappa>0$ and
\item[(b)]   $ B $ is convex
\end{enumerate}
then
\begin{align}\label{eq:AP-convex}
\dist{x_+}{ S }\leq\sqrt{1-\gamma^2}~\dist{x} {S},\quad\forall x_+\in P_B x,~\forall x\in A\cap \Ball_\delta(\hat x)
\end{align}
where $\gamma=1/\kappa$.
\end{mycor}
\begin{proof}
By convexity of $B$ the projector $P_B$ is nonexpansive and it follows that $P_Bx\in \Ball_\delta(\hat x)$ for all $x\in\Ball_\delta(\hat x).$
Saying that $B$ is convex equivalent to saying that $B$ is  $(0,+\infty)$-regular and hence $\relconst=0$ in Theorem \ref{P:epsilondeltareg}.
\end{proof}
\begin{mycor}[linear convergence of MAP]\label{cor:MAP convex}
Let $ A, B $ be closed nonempty subsets of $\E$ and let the collection $\{A,B\}$ be \emph{locally linearly regular} 
at $\hat x\in S := A\cap B $ on $\Ball_\delta(\hat x)$ with regularity modulus $\kappa>0$. Define $\gamma\equiv 1/\kappa$
and let $x_0\in A$.
Generate the sequence $\{x_n\}_{n\in\N}$ by
\begin{equation}
   x_{2n+1}\in P_{ B}x_{2n}~\textup{ and }~ x_{2n+2}\in P_{ A }x_{2n+1}\quad\forall n=0,1,2,\dots.
\end{equation}
\begin{enumerate}
 \item[(a)] If $ A$ and $ B $ are $(\varepsilon,\delta)-$subregular at $\hat x$ with respect to  $S$ and 
$\relconst\equiv 2\varepsilon+2\varepsilon^2\leq\gamma^2$, then
\begin{align}
 \dist{x_{2n+2}} S \leq (1-\gamma^2+\relconst)\dist {x_{2n}}  S\quad\forall n=0,1,2,\dots
\end{align}
for all $x_0\in \Ball_{\delta/ 2}(\hat x)\cap A$.
\item[(b)] If $ A$ is $(\varepsilon,\delta)-$subregular with respect to  $S$, $ B $ is convex and 
$\relconst\equiv 2\varepsilon+2\varepsilon^2\leq(2\gamma-\gamma^2)/(1-\gamma^2)$, then
\begin{align}
 \dist{x_{2n+2}} S \leq \sqrt{1-\gamma^2+\relconst}\sqrt{1-\gamma^2}\dist {x_{2n}}  S\quad\forall n=0,1,2,\dots ,
\end{align}
for all $x_0\in \Ball_{\delta/ 2}(\hat x)\cap A$.

\item[(c)] If $ A$ and $ B $ are convex, then
\begin{align}
 \dist{x_{2n+2}} S \leq (1-\gamma^2)\dist {x_{2n}}  S \quad\forall n=0,1,2,\dots
\end{align}
for all $x_0\in \Ball_{\delta}(\hat x)\cap A$.
\end{enumerate}
\end{mycor}
\begin{proof}

$a)$ First one has to show that all iterates remain close to $\hat x$ for $x_0$ close to $\hat x$, that is,
we have to show that all iterates remain in the set $U$ defined by \eqref{eq:AP:regionofconvergence}. 
Note that for any $x_0\in\Ball_{\delta/2}(\hat x)$, and $x_1\in P_B x_0$ one has 
 \[
\norm{x_0-x_1}=\dist {x_0} B\leq \norm{x_0-\hat x}.
 \]
 since $\hat x\in B$.
 Thus
\begin{align} 
\norm{ x_1-\hat x }\leq {\norm{x_0-x_1}}+{\norm{x_0-\hat x}}
\leq{\norm{x_0-\hat x}}+{\norm{x_0-\hat x}}\leq\delta,
\end{align} 
which shows that $P_B x_0\subset\Ball_\delta(\hat x)$, $\forall x_0\in\Ball_{\delta/2}(\hat x)$.
One can now apply Theorem \ref{P:epsilondeltareg} to conclude that 
\begin{align}
 \dist {x_1}{S}\leq\sqrt{1-\gamma^2+\relconst}~\dist{x_0}{S}.
\end{align} 
The last equation then implies that $x_1\in\Ball_{\delta/2}(\hat x)$ as long as $\gamma^2\geq\relconst$
and therefore the same argument 
can be applied to $x_1$ to conclude that
\begin{align} \label{eq:prfMAP1}
 \dist {x_2}{S}\leq\sqrt{1-\gamma^2+\relconst}~\dist{x_1}{S}.
\end{align}
Combining the last two equations  $(a)$ then follow by induction.

$b)$ 
Applying Corollary \ref{t:GPR65} yields
\begin{align} 
 \dist {x_1}{S}\leq\sqrt{1-\gamma^2}~\dist{x_0}{S}
\end{align} 
and analogous to $a)$ note that \eqref{eq:prfMAP1} is still valid for $\relconst\leq(2\gamma-\gamma^2)/(1-\gamma^2)$. 
By $\relconst\leq(2\gamma-\gamma^2)/(1-\gamma^2)$ it follows that 
\begin{align}
\sqrt{1-\gamma^2+\relconst}\sqrt{1-\gamma^2}
\leq& \sqrt{1-\gamma^2+(2\gamma-\gamma^2)/(1-\gamma^2)}\sqrt{1-\gamma^2}\\
\leq& \sqrt{1-2\gamma +\gamma^2+(2\gamma-\gamma^2)}\\
\leq&1
\end{align}
and therefore by induction $b)$.

$c)$ is an immediate consequence of Corollary \ref{t:GPR65}.
\end{proof}

\subsection{Linear Convergence of Douglas-Rachford}
We now turn to the Douglas-Rachford algorithm.  This algorithm is notoriously difficult to analyze and our 
results reflect this in considerably more circumscribed conditions than are required
for the MAP algorithm.  Nevertheless, to our knowledge the following convergence results are
the most general to date.  The first result gives sufficient conditions for the coercivity
condition \eqref{eq:Tcoerciv} to hold.   
\begin{mylemma}\label{t:Regina Spector}
Let the collection of closed subsets $A, B$ of $\E$ be locally linearly regular at $\hat x\in S\equiv A\cap B$ 
on $\Ball_\delta(\hat x)$ with constant $\kappa>0$ for some $\delta>0$.
 Suppose further that $B$ is a subspace and 
that for some constant $c\in (0,1)$ the following condition holds:
\begin{align}
\begin{matrix}
x\in\Ball_\delta(\hat x), ~y= P_Bx,\\
z\in P_A(2y-x) 
\end{matrix}
\quad\mbox{ and }\quad
\left. \begin{matrix}
  u\in N_A (z)\cap \Ball\\
  v\in N_B (y)\cap \Ball\\
 \end{matrix}\right\}\quad\Rightarrow \quad\langle u,v\rangle\geq -c.\label{DRcond1}
\end{align}
Then $T_{DR}$ satisfies
\begin{align}\label{eq:DRcoerciv}
\norm{x-x_+}\geq \frac{\sqrt{1-c}}{\kappa} \dist{x}{S}\quad\forall x_+\in T_{DR}x,~\forall ~x\in U,
\end{align}
where
\begin{align}\label{eq:UconvAAR}
U\subset \left\{ x\in B_\delta(\hat x)\middle |~ P_A R_B x\subset\Ball_\delta(\hat x)\right\} .
\end{align}

\end{mylemma}

\begin{proof}
In what follows we will use the notation $R_Bx$ for $2y-x$ with $y= P_Bx$ 
which is unambiguous, if a slight abuse of notation, since $B$ is convex.  
We will use \eqref{DRcond1} to show that for all $x\in U$ with $y= P_Bx$ and 
$z\in P_AR_Bx$:
\begin{align}\label{eq:DR:prf1}
\begin{aligned}
 \norm{x-x_+}^2&=\norm{z-y}^2\\
&\geq (1-c) \left[\norm{z-R_B x}^2 +\norm{R_Bx - y}^2\right]
\end{aligned}
\end{align}
We will then show that, for all $x\in U$ with $y= P_Bx$ and 
$z\in P_AR_Bx$ 
\begin{align}\label{eq:DR:prf2}
\norm{z-R_Bx}^2 +\norm{R_B x - y}^2\geq \frac{1}{\kappa^2}\dist{R_Bx}{S}^2. 
\end{align}
Combining inequalities \eqref{eq:DR:prf1} and \eqref{eq:DR:prf2} yields
\begin{align} 
\norm{x-x_+}^2\geq\frac{1-c}{\kappa^2}\dist{R_Bx}{S}^2
\quad\forall~ x\in U.
\end{align} 
Let $\tilde x\in P_S(R_B x)$ and note that $\dist{R_Bx}{S}=\|R_Bx-\tilde x\|$.  
Since $B$ is a subspace, by \eqref{eq:Rdist}  one has $\dist{R_Bx}{S}=\|R_Bx-\tilde x\|=\|x-\tilde x\|$.
Moreover, $\norm{x-\tilde x}\geq\min_{y\in S}\norm{x-y} = \dist{x}{S}$, hence
\begin{align} 
 \norm{x-x_+}\geq \frac{\sqrt{1-c}}{\kappa}\dist{x}{S}
\quad\forall~ x\in U
\end{align} 
as claimed.

What remains is to prove \eqref{eq:DR:prf1} and \eqref{eq:DR:prf2}
for all $x\in U$ with $y= P_Bx$ and 
$z\in P_AR_Bx$.

\noindent {\em Proof of \eqref{eq:DR:prf1}}.   
Using Lemma \ref{lemma:AAR} equation \eqref{eq:AAR2} one has for $x\in\Ball_\delta(\hat x)$ 
with $y= P_Bx$ and 
$z\in P_AR_Bx$ 
\begin{align}
\norm{x-x_+}^2
=&\norm{z-y}^2\notag\\
=&\norm{z-R_B x+R_B x -y}^2\notag\\
=&\norm{z-R_B x}^2 +\norm{R_B x - y}^2
+2\langle \underbrace{z-R_B x}_{\in  -N_A (z)},\underbrace{R_B x - y}_{=y-x\in - N_B (y)}\rangle\notag\\
\stackrel{\eqref{DRcond1}}\geq&\norm{z-R_B x}^2 +\norm{R_B x - y}^2
-2c\norm {z-R_B x}\norm{R_B x - y}\notag\\
=&(1-c)\left[\norm{z-R_B x}^2 +\norm{R_B x - y}^2\right]\notag\\
&+c\left[\norm{z-R_B x}^2 -2\norm {z-R_B x}\norm{R_B x - y}+\norm{R_B x - y}^2\right]\notag\\
=&(1-c) \left[\norm{z-R_B x}^2 +\norm{R_B x - y}^2\right]\notag\\
&+c\left[\norm{z-R_B x}-\norm{R_B x - y}\right]^2\label{eq:tight1}\\
\geq& (1-c) \left[\norm{z-R_B x}^2 +\norm{R_B x - y}^2\right]\notag.
\end{align}
\hfill $\triangle$\\

\noindent {\em Proof of \eqref{eq:DR:prf2}}.   
First note that if $x\in \Ball_{\delta}(\hat x)$, since $B$ is a 
subspace, by equation \eqref{eq:Rdist} $R_Bx\subset\Ball_{\delta}(\hat x)$
and by convexity of $\Ball_{\delta}(\hat x)$ it follows that  $y=P_B x\subset \Ball_{\delta}(\hat x)$ and 
hence \eqref{DRcond1} is localized to $\Ball_\delta(\hat x)$ (to which the dividends of linear regularity of the 
intersection extend) as long as $P_AR_Bx\subset\Ball_\delta(\hat x)$, 
that is, as long as $x\in U$.  By definition of the projector $\norm{R_B x - y}\geq\norm{R_B x - P_B (R_B x)}$.
Local linear regularity at $\hat x$ with radius $\delta$ and constant $\kappa$ 
yields for $x\in U$
with $y= P_Bx$ and 
$z\in P_AR_Bx$ 
\begin{align}
\norm{z-R_B x}^2 +\norm{R_B x - y}^2\geq&\norm{z-R_B x}^2 +\norm{R_B x - P_B R_Bx}^2\notag\\
=&\dist {R_B x}{A}^2+\dist{R_B x} {B}^2\notag\\
\stackrel{\eqref{eq:locallinear}}\geq& \frac{1}{\kappa^2} \dist {R_B x}{S}^2
\label{eq:tight2}
\end{align}
This completes the proof of \eqref{eq:DR:prf2}
and the Theorem.
\end{proof}

\begin{myremark}
The coercivity constant in equation \eqref{eq:DRcoerciv} is not tight, even if $\kappa$ were chosen to 
the the regularity modulus of the intersection.
Remember, by linearity, that $y=P_Bx=P_BR_Bx$.
In line \eqref{eq:tight1} we use the inequality 
$\left[\norm{z-R_B x}-\norm{R_B x - y}\right]^2\geq0$ while we use 
$\norm{z-R_B x}^2+\norm{R_B x - y}^2\geq\max\{\norm{z-R_B x}^2,\norm{R_B x - y}^2\}$ in line \eqref{eq:tight2}.
If the first inequality is tight then this is the worst possible result in the second inequality, since $\norm{z-R_B x} =\norm{R_B x - y}$, 
that is, this second inequality is satisfied not only strictly, but the inequality is as large as it can possibly be. 
On the other hand if the second inequality is tight this means $\norm{z-R_B x}^2=0$ \emph{or} $\norm{R_B x - y}^2=0$ 
and this means that the first inequality is strict.  In any event, it is impossible to achieve 
equality in the argumentation of the proof.  This is a technical limitation of the logic of the proof and does not 
preclude improvements. 
\hfill$\Box$\end{myremark}

Lemma \ref{t:Regina Spector} with the added assumption of $(\epsilon,\delta)$-regularity of the nonconvex 
set yields local linear convergence of the Douglas-Rachford algorithm in this special case. 
\begin{mythm}\label{DRsub}
Let the collection of closed subsets $A,B$ of $\E$ be locally linearly regular at $\hat x\in S\equiv A\cap B$ 
on $\Ball_\delta(\hat x)$ with constant $\kappa>0$ for some $\delta>0$.
 Suppose further that $B$ is a subspace and that $A$ is $(\varepsilon,\delta)$-regular at $\hat x$ with respect to  $S$.
Assume that for some constant $c\in (0,1)$ the following condition holds:
\begin{align}\label{e:Sigur}
\left. \begin{matrix}
  z\in A\cap   \Ball_\delta(\hat x),& u\in N_A(z)\cap  \Ball\\
  y\in B\cap   \Ball_\delta(\hat x),& v\in N_B(y)\cap \Ball\\
 \end{matrix}\right\}\quad\Rightarrow \quad\langle u,v\rangle\geq -c.
\end{align}
If $x\in \Ball_{\delta/2}(\hat x)$ then
\begin{align}
 \dist{x_+}{S}\leq\sqrt{1+\relconst-\eta}~ \dist x {S}\quad\forall~x_+\in T_{DR}x
\end{align}
with $\eta:=\frac{(1-c)}{\kappa^2}$ and $\relconst=2\varepsilon+2\varepsilon^2$.
\end{mythm}
\begin{proof}
First one has to show requirement \eqref{eq:UconvAAR}.
Since $\hat x\in A\cap B$ note that for any $x\in\Ball_{\delta/2}(\hat x)$ for all $z\in P_A R_B x$
by Definition $\norm{z-R_B x}=\dist{R_B x}{A}\leq\norm{R_B x-\hat x}$ and by \eqref{eq:Rdist}
$\norm{R_B x-\hat x}=\norm{x-\hat x}$ holds.
 This now implies
\begin{align} 
\norm{z -\hat x }\leq \norm{z- R_B x}+\norm{R_B x-\hat x}
\leq2{\norm{x-\hat x}}\leq\delta,
\end{align} 
and therefore $z\in \Ball_\delta(\hat x)$.

Now for $B$ a subspace \eqref{e:Sigur} and \eqref{DRcond1} are equivalent, and so by Lemma \ref{t:Regina Spector} the coercivity condition \eqref{eq:Tcoerciv}
\begin{equation}
\norm{x-x_+}\geq \frac{\sqrt{1-c}}{\kappa} \dist{x}{S}
\end{equation}
is satisfied on $\Ball_{\delta/2}(\hat x)$.  Moreover, since $A$ is $(\varepsilon,\delta)$-regular and $B$ is $(0,\infty)$-regular, 
by Theorem \ref{t:Velvet Underground} $T_{DR}$ is $(S,\relconst)$-firmly nonexpansive 
with $\relconst=2\varepsilon(1+\varepsilon)$, that is 
($\forall x\in  \Ball_{\delta/2}(\hat x)$) ($\forall x_+\in T_{DR} x$) ($\forall \point\in S$) 
\begin{align}
\begin{aligned}
 \norm{x_+-\point}^2+\norm{x -x_+}^2\leq\left(1+\relconst\right)\norm{x-\point}^2.
\end{aligned}
\end{align}
Lemma \ref{Tconv} then applies to yield ($\forall x\in  \Ball_{\delta/2}(\hat x)$) ($\forall x_+\in T_{DR} x$) 
\begin{align} 
\begin{aligned}
 \dist{x_+}{S}\leq\sqrt{1+\relconst-\eta}\ \dist{x}{S}\\
\end{aligned}
\end{align} 
where $\eta\equiv \frac{1-c}{\kappa^2}$.
\end{proof}
The next lemma establishes sufficient conditions under which \eqref{e:Sigur} holds.
\begin{mylemma}[\cite{LLM} Theorem 5.16]\label{Lemma:LLM}
 Assume $B\subset\E$ is a subspace and that $A\subset\E$ is closed and super-regular at $\hat x\in A\cap B$.
If the collection $\{A,B\}$ is strongly regular at $\hat x$,  then there is a $\delta>0$ 
and a constant 
$c\in (0,1)$ such that \eqref{e:Sigur} holds on $\Ball_\delta(\hat x)$.
\end{mylemma}
\begin{proof}
 Condition \eqref{DRcond1} can be shown using \eqref{linearregularcone}. For more details see \cite{LLM}.
\end{proof}
We summarize this discussion with the following convergence result for the Douglas-Rachford algorithm in the case of 
an affine subspace and a super-regular set.  
\begin{mythm}\label{thm:DRmainresult}
    Assume $B\subset\E$ is a subspace and that $A\subset\E$ is closed and super-regular at $\hat x\in S\equiv A\cap B$.
If the collection $\{A,B\}$ is strongly linearly regular at $S$,  then there is a $ \delta>0$ such that, 
\begin{align}\label{eq:requirementDR}
 \frac{(1-c)}{\kappa^2}> 2\varepsilon+2\varepsilon^2
\end{align}
and hence
\begin{align}
 \dist{x_+}{S}\leq\tilde c ~ \dist x {S}\quad\forall~ x_+\in T_{DR}x,
\end{align}
with $\tilde c=\sqrt{1+2\varepsilon+2\varepsilon^2-\frac{(1-c)}{\kappa^2}}<1$ 
for all $x\in\Ball_{\delta/2}(\hat x)$. 
\end{mythm}
\begin{proof}
Strong regularity of the collection implies linear regularity with constant $\kappa$ on $\Ball_{\delta_1}(\hat x)$ 
(see Remark \ref{remark:stronglinear}).
Lemma \ref{Lemma:LLM} guaranties the existence of constants $\delta_2>0$ and $c\in (0,1)$ such that \eqref{e:Sigur} holds on $\Ball_{\delta_2}(\hat x)$.
Now, by super-regularity at $\hat x$, for any $\varepsilon$ there exists a $\delta_3$ such that $A$ is $(\varepsilon,\delta_3)$-subregular at $\hat x$.  In other
words, for $c$ and $\kappa$ determined by the regularity of the collection $\{A, B\}$ at $\hat x$, we can always choose $\varepsilon$ (generating a 
corresponding $\delta_3$ radius) so that \eqref{eq:requirementDR} is satisfied on $\Ball_{\delta_3}(\hat x)$.  
Then for $\delta\equiv\min\left\{\delta_1,\delta_2,\delta_3\right\}$, the requirements of Theorem \ref{DRsub} are 
satisfied on $\Ball_\delta(\hat x)$, which completes the proof of linear convergence on
$\Ball_{\delta/2}(\hat x)$.
\end{proof}

\begin{myremark}
The example Example \ref{examples} \eqref{ex:Borwein} has been studied by Borwein and coauthors \cite{BorweinAragon,BorweinSims} where 
they achieve global characterizations of convergence with rates.  Our work does not directly overlap with \cite{BorweinAragon,BorweinSims}
since our results are local, and the order of the reflectors is reversed:  we must reflect first across the subspace, then reflect across the 
nonconvex set; Borwein and coauthors reflect first across the circle.   
\hfill$\Box$
\end{myremark}

\subsection{Douglas-Rachford on Subspaces}
We finish this section with the fact that strong regularity of the intersection is {\em necessary}, not just sufficient for 
convergence of the iterates of the Douglas-Rachford algorithm to the intersection in the affine case. 

\begin{mycor}\label{cor:subspaces}
 Let $A,B$ be two affine subspaces with $A\cap B\neq\emptyset$. Douglas-Rachford converges for any starting point $x_0\in \E$ 
with linear rate to the intersection $A\cap B$ if and only if $A^\perp\cap B^\perp=\left\{0\right\}$.
\end{mycor}
\begin{proof}
Without loss of generality for $\hat x\in A\cap B$  by shifting the subspaces by $\hat x$ we consider the case of linear subspaces.
By \eqref{linearregularcone},  on subspaces  the condition $A^\perp\cap B^\perp=\left\{0\right\}$ is equivalent to 
strong regularity of the collection $\{A,B\}$ .

If the intersection is strongly regular and $A$ and $B$ are subspaces, then the requirements of Theorem \ref{thm:DRmainresult} are globally satisfied,
so Douglas-Rachford converges with linear rate
\begin{align} 
\tilde c =\sqrt{1-\frac{(1-c)}{\kappa^2}}<1
\end{align} 
where $c\in [0,1)$ (compare \eqref{DRcond1}) now becomes
\begin{align}\label{e:constantsubspaces}
 c=\max \langle u,v\rangle,\qquad u\in A^\perp,~\norm{u}=1,~ v\in B^\perp,~\norm{v}=1,
\end{align}
and $\kappa$ is an associated \emph{global} constant of linear regularity (see Theorem \ref{t:Gruene Libanon}).

On the other hand for $\hat x\in A\cap B$ by \cite[Thm 3.5]{BauComLuke} we get the characterization
 \begin{align} 
\textup{Fix}\, T_{DR}=&(A\cap B)+N_{A-B}(0)\\
=&(A\cap B)+(N_A(\hat x)\cap -N_B(\hat x))\\
=&(A\cap B)+A^\perp\cap B^\perp.
\end{align} 
and so the fix point set of $T_{DR}$ does not coincide with the intersection unless the collection $\{A,B\}$ is 
strongly regular.  In other words, if the intersection is not strongly regular, then convergence to the intersection 
cannot be linear, thus proving the reverse implication by the contrapositive.
\end{proof}

\begin{myremark}[Friedrichs angle \cite{Friedrichs}]
We would like to make a final remark about connection between the notion of the angle of the sets at 
the intersection and the regularity of the collection of sets at points in the intersection.  The operative
notion of angle is the \emph{Friedrichs angle}.  For two subspaces $A$ and $B$ the 
\emph{Friedrichs angle} is the angle $\alpha(A,B)$ in $[0,\pi/2]$ whose cosine is defined by
\begin{align}\label{e:Friedrichs}
c_{F}(A,B):=\sup \left\{ |\langle x, y\rangle|~\middle|
\begin{matrix}
~ x\in A\cap(A\cap B)^\perp,~\norm{x}\leq1,\\~y\in B\cap(A\cap B)^\perp,~\norm{y}\leq 1. 
\end{matrix}\right\}
\end{align}
The Friedrichs angle being less than $1$ is not sufficient for convergence of Douglas-Rachford.
This can be seen by example \ref{examples} \eqref{ex:lines3}.
The Friedrichs angle in this example is the same as in \ref{examples} \eqref{ex:lines2}, but for $x_0\notin \R^2\times\{0\}$  the Douglas-Rachford algorithm does not converge to $\{0\}=A\cap B$.
Another interesting observation is that if, on the other hand, $A^\perp\cap B^\perp=\{0\}$ 
then $(A^\perp\cap B^\perp)^\perp=\E$ which implies that $c_{F}(A^\perp,B^\perp)$ 
coincides with \eqref{e:constantsubspaces} and by
\cite[Theorem 2.16]{Deutschangle} $c_{F}(A^\perp,B^\perp)$ then coincides with $c_{F}(A,B)$. 
So if the Douglas-Rachford algorithm on subspaces converges linearly, then the rate of convergence is 
dependent on the Friedrichs angle.
A detailed analysis regarding the relation between the Friedrichs angle and linear convergence of MAP can be 
found in \cite{DeutschCPA01,DeutschCPA03}.
\hfill$\Box$\end{myremark}

\section{Concluding Remarks}

In the time that has passed since first submitting our manuscript for publication 
we learned about an optimal linear convergence result for Douglas-Rachford applied to 
$\ell_1$ optimization with an affine constraint using different techniques \cite{DemanetZhang13}.  
The modulus of linear regularity does not recover optimal convergence results (see Remark \ref{remark:overview}),
but we suspect this is an artifact of our proof technique.  The question remains whether there is a quantitative primal 
definition of a angle between two sets that recovers the same results for MAP as \cite{BLPW2}.
This could also be useful to achieve optimal linear convergence results for Douglas-Rachford in general.
Also, as we noted in the introduction, it is well-known that the fixed-point set of the Douglas-Rachford operator is in 
general bigger than the intersection of the sets, and Corollary \ref{cor:subspaces} stating that the iterates converge 
to the intersection if and only if the collection of sets is strongly regular is a consequence of this.  In the 
convex case, the {\em shadows} of the iterates still converge.  We leave a fuller investigation of the 
shadows of the iterates and the angles between the 
sets at the intersection in the nonconvex setting to future work.  

Another direction of future work will be to extend this analysis more generally to fixed point mappings
built upon functions and more general set-valued mappings, but also in particular  {\em proximal} operators and reflectors.  
The generality of our approach makes such extensions
quite natural.  Indeed, local linear regularity of collections of sets can be shown to be related to 
{\em strong metric subregularity} of set-valued mappings which guarantees that the condition \eqref{eq:Tcoerciv}
of Lemma \ref{Tconv} is satisfied.  Of course, the difficulty remains to show that the 
the mappings are indeed metrically subregular.  

\section*{Acknowledgments} The authors gratefully acknowledge the support of DFG-SFB grant 755-TPC2.

\addcontentsline{toc}{section}{Literature}

\end{document}